\documentclass[aps,secnumarabic,nobalancelastpage,amsmath,amssymb,
nofootinbib]{revtex4}

\usepackage{amscd}
\usepackage{amssymb}
\usepackage{amsmath,amsthm,amssymb,xypic,hyperref,amsfonts}
\usepackage[all, knot]{xy}
\usepackage{graphicx}
\xyoption{arc}
\usepackage{epsfig} 
\usepackage{color}
\usepackage{chapterbib}
\usepackage{graphics}      
\usepackage{graphicx}      
\usepackage{longtable}     
\usepackage{bm}            

\numberwithin{equation}{section}
\newtheorem{thm}[equation]{Theorem}\newtheorem*{thm*}{Theorem}
\newtheorem{cor}[equation]{Corollary}\newtheorem*{cor*}{Corollary}
\newtheorem{lem}[equation]{Lemma}\newtheorem*{lem*}{Lemma}
\newtheorem*{prop*}{Proposition}
\newtheorem*{defn*}{Definition}
\newtheorem*{ex*}{Example}
\newtheorem*{rmk*}{Remark}

\newcommand{\bra}[1]{\left\langle #1 \right|}
\newcommand{\ket}[1]{\left|#1\right\rangle}
\newcommand{\braket}[2]{\left\langle#1 |  #2\right\rangle}

\begin{document}
\title{The Inverse of a Nearly-banded Matrix}
\author{Ruitian Lang}
\affiliation{Department of Physics, MIT, Cambridge, MA02139}
\begin{abstract}
A quantitative form of the Nullity Theorem is presented, which establishes a linear relation between the singular values of the two submatrices involved in the theorem up to the first order. The theorem is then extended to function spaces and a corresponding form in infinite dimension is discussed.
\end{abstract}

\maketitle
\section{Introduction}
The Nullity Theorem is concerned with submatrices of a block matrix T and its inverse. The theorem says that complementary submatrices have the same nullity -- a basic result in linear algebra that deserves to be better known. This paper responds to a problem posed by Gilbert Strang, to find estimates of singular values when they are nearly but not exactly zero. Roughly speaking, if a submatrix of $T$ \emph{almost} has nullity $k$ (meaning that its $(k+1)$st smallest singular value is an infinitesimal), then the complementary submatrix of $T^{-1}$ almost has this nullity.

At the end we consider extensions to operators on function spaces. The Nullity Theorem does not explicitly involve the orders of the submatrices (unless it is stated as an equality for their ranks). It extends to infinite dimensions and we look for a corresponding quantitative form.

Banded matrices arise from various contexts in applied mathematics
and physics, especially the tridiagonal matrices. They are of
interest partly because the number of nonzero entries of a banded
matrix is linear to its size, which makes their manipulations
much simpler than a full matrix. However, the inverse of a banded
matrix is not banded. Therefore, it requires more work to reduce the
complexity of the computations involving its inverse. Fortunately, the
classical theorem, which we will present in the
next section, asserts that the off-diagonal submatrices of the
inverse matrix have low ranks. This theorem makes our calculation
much easier: if an $n\times n$ matrix $A$ is of rank $k$, then
$A=BC$ for some $n\times k$ matrix $B$ and some $k\times n$ matrix
$C$.

The matrices that we encounter in practice may not be
banded, but their entries out of a band are small. Thus we may still
want to use banded matrices to simplify the computation.
For example, in solid state physics, the tight-binding model produces a tridiagonal Hamiltonian, but taking into account the ``long-range'' interactions yields an approximate tridiagonal Hamiltonian \cite{solid}.
If we focus on the off-diagonal blocks of the inverse matrix, this
process can be interpreted as the approximation of the original
blocks by low rank matrices. 

In this paper, we will study the error
that arises from this approximation.

\section{The nullity theorem}
The nullity theorem has been proved in many articles, including \cite{Strang} and \cite{Hugo}. We will state this
theorem in a notation convenient for us. The nullity of a matrix $B$ means the dimension of its kernel.
\begin{thm}\label{nullity}
If $\left(
     \begin{array}{cc}
       A & B \\
       G & D \\
     \end{array}
   \right)^{-1}=\left(
                  \begin{array}{cc}
                    E & C \\
                    H & F \\
                  \end{array}
                \right)$, then $\mathrm{dimker}B=\mathrm{dimker}C$. In other words, $\mathrm{nullity}(B)=\mathrm{nullity}(C)$.
\end{thm}

We say that an $M$ by $M$ matrix $K=(k_{ij})$ is a banded matrix of \emph{bandwidth} $p$, if $k_{ij}=0$ whenever $|i-j|>p$.
In such a case, we write $K$ and $K^{-1}$ in the block form as
$K=\left(
     \begin{array}{cc}
       A & B \\
       G & D \\
     \end{array}
   \right)$ and $K^{-1}=\left(
     \begin{array}{cc}
       E & C \\
       H & F \\
     \end{array}
   \right)$, where $A$ is $n\times (n+p)$ and $C$ is
$(n+p)\times(M-n)$. Since the bandwidth of $K$ is $p$, $B=0$.
Hence, dimker$C$=dimker$B=M-n-p$, so $\mathrm{rank}C=p$, which is
independent of the size of $K$.
\begin{cor}\label{rank}
If $K$ is a banded matrix of bandwidth $p$, and $C$ is a submatrix
of $K^{-1}$ above the $p$th subdiagonal, then $\mathrm{rank}C\leq p$.
\end{cor}
This is the precise statement of the low off-diagonal rank property
of $K^{-1}$. We will use the notation of this corollary throughout.
By symmetry, a similar result holds for the lower
off-diagonal submatrices of $K^{-1}$.\\
\section{The main theorem}
Now we consider the case in which $\|B\|$ is small but nonzero. We
want to replace $C$ by some matrix $L$ of rank $p$. The first
attempt is quite intuitive. We treat the non-zero block $B$ as a
perturbation and look at the unperturbed matrix $K_{0}=\left(
                                                         \begin{array}{cc}
                                                           A & 0 \\
                                                           G & D \\
                                                         \end{array}
                                                       \right)$ and
the upper right block of its inverse, $C_{0}$, which is the first
attempt we make to approximate $C$. Let $\Delta K=\left(
                                                         \begin{array}{cc}
                                                           0 & B \\
                                                           0 & 0 \\
                                                         \end{array}
                                                       \right)$ and $\epsilon=\|\Delta K\|$. When $\epsilon<1/\|K_{0}^{-1}\|$, we have
\begin{equation}
K^{-1}=(K_{0}+\Delta
K)^{-1}=K_{0}^{-1}\sum_{n=0}^{\infty}(-1)^{n}(K_{0}^{-1}\Delta
K)^{n}
\end{equation}
Therefore,
\begin{eqnarray*}
\|C-C_{0}\|&\leq&\|K^{-1}-K_{0}^{-1}\|\\
&\leq&\|K_{0}^{-1}\|\sum_{n=1}^{\infty}(-1)^{n}\|K_{0}^{-1}\Delta
K\|^{n}\\
&=&\frac{\|K_{0}^{-1}\|^{2}\epsilon}{1-\|K_{0}^{-1}\|\epsilon}
\end{eqnarray*}
This is a nice estimate in the sense that it bounds the error
$\|C-C_{0}\|$ to a function of $\epsilon$ no matter what the
perturbation $B$ is.
However, $C_{0}$ is not necessarily the best approximation. Recall
that\cite{svd}
\begin{equation}\label{svd}
\sigma_{k}(C)=\inf\{\|C-L\|:\mathrm{rank}L\leq k-1\}
\end{equation}
Hence, the minimum error that arises from replacing $C$ by a
rank-$p$ matrix is $\sigma_{p+1}(C)$, and the SVD also determines $L$
that minimizes the error. We are concerned with the upper bound of
$\sigma_{p+1}(C)$ when $B$ varies. Let $L$ denote the best
approximation of rank $p$ to $C$ and $J=\left(
     \begin{array}{cc}
       E & L \\
       H & F \\
     \end{array}
   \right)$. Then we have the following
\begin{lem} If $\epsilon=\|B\|$ is sufficiently small, then $J$ is
invertible, and $J^{-1}=\left(
     \begin{array}{cc}
       \tilde{A} & 0 \\
       \tilde{G} & \tilde{D} \\
     \end{array}
   \right)$ for some $\tilde{A},\tilde{G},\tilde{D}$. Furthermore,
$B=A(L-C)\tilde{D}$ and $A-\tilde{A}=A(L-C)\tilde{G}$.
\end{lem}
\begin{proof}
$J$ is invertible because
\[\|J-K^{-1}\|=\|L-C\|=\sigma_{p+1}(C)\leq\frac{\|K_{0}^{-1}\|^{2}\epsilon}{1-\|K_{0}^{-1}\|\epsilon},\]
which is less than $1/\|K\|$ if $\epsilon$ is sufficiently small.
Here we have used our first estimate. Let
\[J^{-1}=\left(
     \begin{array}{cc}
       \tilde{A} & \tilde{B} \\
       \tilde{G} & \tilde{D} \\
     \end{array}
   \right).\]Then $\tilde{B}=0$ by the nullity theorem.\\
To prove the second statement, we notice that
\begin{eqnarray*}
\left(
     \begin{array}{cc}
       A & B\\
       G & D \\
     \end{array}
   \right)&=&\left(
     \begin{array}{cc}
       A & B \\
       G & D \\
     \end{array}
   \right)\left(
     \begin{array}{cc}
       E & L \\
       H & F \\
     \end{array}
   \right)\left(
     \begin{array}{cc}
       \tilde{A} & 0 \\
       \tilde{G} & \tilde{D} \\
     \end{array}
   \right)\\
&=&\left(
     \begin{array}{cc}
       AE+BH & AL+BF \\
       GE+DH & GL+DF \\
     \end{array}
   \right)\left(
     \begin{array}{cc}
       \tilde{A} & 0 \\
       \tilde{G} & \tilde{D} \\
     \end{array}
   \right)
\end{eqnarray*}
However,
\[\left(
     \begin{array}{cc}
       A & B \\
       G & D \\
     \end{array}
   \right)\left(
     \begin{array}{cc}
       E & C \\
       H & F \\
     \end{array}
   \right)=\left(
     \begin{array}{cc}
       I & 0 \\
       0 & I \\
     \end{array}
   \right),\]
so $AE+BH=I$ and $AL+BF=A(L-C)$ since $AC+BF=0$. Hence,
\[\left(
     \begin{array}{cc}
       A & B \\
       G & D \\
     \end{array}
   \right)=\left(
     \begin{array}{cc}
       I & A(L-C) \\
       * & * \\
     \end{array}
   \right)\left(
     \begin{array}{cc}
       \tilde{A} & 0 \\
       \tilde{G} & \tilde{D} \\
     \end{array}
   \right)=\left(
     \begin{array}{cc}
       \tilde{A}+A(L-C)\tilde{G} & A(L-C)\tilde{D} \\
       * & * \\
     \end{array}
   \right).\]
Therefore, $B=A(L-C)\tilde{D}$ and $A-\tilde{A}=A(L-C)\tilde{G}$.
\end{proof}
Now we can state and prove our main theorem.
\begin{thm}\label{main}
For sufficiently small $\epsilon$,
\[\sup_{\|B\|\leq\epsilon}\frac{\sigma_{p+1}(C)}{\|B\|}=\frac{1}{\sigma_{p}(A)\sigma_{M-n-p}(D)}+O(\epsilon)\]
\end{thm}
\begin{proof}
We first show that
\[\frac{\sigma_{n}(A)\sigma_{M-n-p}(\tilde{D})\sigma_{p+1}(C)}{1+\sigma_{p+1}(C)\|\tilde{G}L\|/\sigma_{p}(C)}\leq\|B\|\]. The best
approximation $L$ is given by the SVD as follows: if
\begin{equation}
C=Q\left(
     \begin{array}{ccc}
       \sigma_{1}(C) &  & 0 \\
        & \ddots &  \\
       0 &  & \sigma_{k}(C) \\
     \end{array}
   \right)
P^{-1},
\end{equation}
where $Q$ and $P$ are orthogonal matrices and $k=\min\{n+p,M-n\}$,
then
\begin{equation}
L=Q\left(
     \begin{array}{cccc}
       \sigma_{1}(C) &  & 0 &  \\
        & \ddots &  &  \\
       0 &  & \sigma_{p}(C) &  \\
        &  &  & 0 \\
     \end{array}
   \right)
P^{-1}
\end{equation}
By the lemma, we have
\begin{equation}
B=A(L-C)\tilde{D}=AQ\left(
                      \begin{array}{cccc}
                        0 &  &  &  \\
                         & -\sigma_{p+1}(C) &  & 0 \\
                         &  & \ddots &  \\
                         & 0 &  & -\sigma_{k}(C) \\
                      \end{array}
                    \right)
P^{-1}\tilde{D}
\end{equation}
Since $L\tilde{D}=0$ by construction, $P^{-1}\tilde{D}$ looks like
$\left(
                                                             \begin{array}{c}
                                                               0 \\
                                                               D_{0} \\
                                                             \end{array}
                                                           \right)$,
where $D_{0}$ is an $(M-n-p)\times (M-n-p)$ matrix and
$\sigma_{M-n-p}(D_{0})=\sigma_{M-n-p}(\tilde{D})$. Let $AQ=\left(
                                                     \begin{array}{cc}
                                                       A_{1} & A_{2} \\
                                                     \end{array}
                                                   \right)$, where
$A_{1}$ has $p$ columns and $A_{2}$ is an $n$ by $n$ matrix. By the lemma,
$AL=A(L-C)\tilde{G}L+\tilde{A}L=A(L-C)\tilde{G}L$. We notice that
\[AL=AQ\left(
         \begin{array}{cccc}
           \sigma_{1}(C) &  & 0 &  \\
            & \ddots &  &  \\
           0 &  & \sigma_{p}(C) &  \\
            &  &  & 0 \\
         \end{array}
       \right)=A_{1}\left(
                      \begin{array}{cccc}
                        \sigma_{1}(C) &  & 0 &  \\
                         & \ddots &  & 0 \\
                        0 &  & \sigma_{p}(C) &  \\
                      \end{array}
                    \right).\]
Hence,
\begin{equation}\label{ineq1}
\|A_{1}\|\sigma_{p}(C)\leq\|AL\|\leq\|A(L-C)\|\|\tilde{G}L\|.
\end{equation}
On the other hand,
\[A(L-C)=A_{2}\left(
                \begin{array}{cccc}
                  -\sigma_{p+1}(C) &  & 0 &  \\
                   & \ddots &  &  \\
                  0 &  & \sigma_{k}(C) &  \\
                   &  &  & 0 \\
                \end{array}
              \right).
\]
Therefore,
\begin{equation}\label{ineq2}
\|A(L-C)\|\geq\sigma_{n}(A_{2})\sigma_{p+1}(C)\geq(\sigma_{n}(A)-\|A_{1}\|)\sigma_{p+1}(C).
\end{equation}
We combine \eqref{ineq1} and \eqref{ineq2} to obtain
\[\sigma_{n}(A)-\frac{\|A(L-C)\|}{\sigma_{p+1}(C)}\leq\|A_{1}\|\leq\frac{\|A(L-C)\|\|\tilde{G}L\|}{\sigma_{p}(C)}.\]
Hence,
\[\|A(L-C)\|\geq\frac{\sigma_{n}(A)\sigma_{p+1}(C)}{1+\sigma_{p+1}(C)\|\tilde{G}L\|/\sigma_{p}(C)}.\]
Now we put all things together:
\begin{eqnarray*}
\|B\|&=&\|A(L-C)\tilde{D}\|\\
&\geq&\|A(L-C)\|\sigma_{M-n-p}(D_{0})\\
&\geq&\frac{\sigma_{n}(A)\sigma_{p+1}(C)\sigma_{M-n-p}(\tilde{D})}{1+\sigma_{p+1}(C)\|\tilde{G}L\|/\sigma_{p}(C)}
\end{eqnarray*}
Now we want to find a $B$ such that the equality holds. We have the
following SVDs:
\begin{eqnarray*}
A&=&Q_{1}\left(
           \begin{array}{cccc}
             \sigma_{1}(A) &  &  &  \\
              & \ddots &  & 0 \\
              & & \sigma_{n}(A) &  \\
           \end{array}
         \right)P_{1}^{-1}\\
\tilde{D}&=&Q_{2}\left(
                   \begin{array}{ccc}
                     \sigma_{1}(\tilde{D}) &  &  \\
                      & \ddots &  \\
                      &  & \sigma_{M-n-p}(\tilde{D}) \\
                      & 0 &  \\
                   \end{array}
                 \right)P_{2}^{-1}
\end{eqnarray*}
Let $T=(t_{ij})$ be an $(n+p)\times(M-n)$ matrix, where
\[t_{ij}=\begin{cases}\frac{\sigma_{p+1}(C)}{1+\sigma_{p+1}(C)\|\tilde{G}L\|/\sigma_{p}(C)}, &\text{if $i=n$ and
$j=M-n-p$;}\\
0, &\text{otherwise}.
\end{cases}\]
Then $Q_{1}^{-1}(AP_{1}TQ_{2}\tilde{D})P_{2}$ has only one nonzero
entry,
$\frac{\sigma_{n}(A)\sigma_{p+1}(C)\sigma_{M-n-p}(\tilde{D})}{1+\sigma_{p+1}(C)\|\tilde{G}L\|/\sigma_{p}(C)}$.
Note that $\tilde{D}$ and $T$ are functions of $B$. Let
$\mathcal{B}=\{B\in\mathbb{R}^{n\times(M-n-p)}|\|B\|\leq\epsilon\}$.
Consider the map
\begin{eqnarray*}
f: \mathcal{B}&\to&\mathbb{R}^{n\times(M-n-p)}\\
B&\mapsto &AP_{1}TQ_{2}^{-1}\tilde{D}
\end{eqnarray*}
By the first half of the proof,
\[\|AP_{1}TQ_{2}^{-1}\tilde{D}\|=\frac{\sigma_{n}(A)\sigma_{p+1}(C)\sigma_{M-n-p}(\tilde{D})}{1+\sigma_{p+1}(C)\|\tilde{G}L\|/\sigma_{p}(C)}\leq\|B\|\leq\epsilon\]
Hence, $f(\mathcal{B})\subset\mathcal{B}$. Now we apply Brouwer
fixed point theorem to conclude that there exists
$B_{0}\in\mathcal{B}$ such that
$B_{0}=f(B_{0})=AP_{1}TQ_{2}^{-1}\tilde{D}$. Therefore,
$\|B_{0}\|=\frac{\sigma_{n}(A)\sigma_{p+1}(C)\sigma_{M-n-p}(\tilde{D})}{1+\sigma_{p+1}(C)\|\tilde{G}L\|/\sigma_{p}(C)}.$
So far, we have shown that
\begin{equation}\label{estimate}
\sup_{\|B\|\leq\epsilon}\frac{\sigma_{n}(A)\sigma_{p+1}(C)\sigma_{M-n-p}(\tilde{D})}{\|B\|(1+\sigma_{p+1}(C)\|\tilde{G}L\|/\sigma_{p}(C))}=1
\end{equation}
Now we notice that the difference between quantities with tilde and
those without tilde are of the order $O(\epsilon)$, and
$\sigma_{p+1}(C)$ is also of order $O(\epsilon)$. Therefore,
\[\sup_{\|B\|\leq\epsilon}\frac{\sigma_{p+1}(C)}{\|B\|}=\frac{1}{\sigma_{n}(A)\sigma_{M-n-p}(D)}+O(\epsilon)\]
\end{proof}
This theorem gives the least upper bound of the error in terms of
the smallest singular values of $A$ and $D$, up to the first order
of $\epsilon$. Eq \eqref{estimate} even works for large $\epsilon$,
though its complex form makes it hardly useful.

\section{Extension to Infinite Dimensions}
The continuous counterpart of the singular value estimate may be formulated in the following way. Let $H_{0}$ be an invertible operator acting on functions of single variable, whose Green's function has finite off-diagonal rank. To be specific, the space of functions is taken to be $L^{2}([0,1])$. The symbol $H_{0}$ does not suggest that it be Hermitian(or self-adjoint). Now suppose a perturbation $\epsilon W$ is turned on, where $\|W\|=1$. In a typical problem, $H_{0}$ is a local operator such as the Schr\"{o}dinger operator $-\frac{d^{2}}{dx^{2}}+V(x)$ and $W$ is an integral operator. The equation satisfied by the Green's function is

\begin{equation}
H_{0}G(x,x')+\epsilon\int_{0}^{1} W(x,x'')G(x'',x')dx''=\delta(x-x')
\label{coordGreen}
\end{equation}
We can rewrite the above equation in terms of Dirac notation:
\begin{equation}
H_{0}\ket{\psi}+\epsilon W\ket{\psi}=\ket{x'}
\label{DiracGreen}
\end{equation}
Here, $\ket{\psi}$ is some function and $\braket{x}{\psi}=\psi(x)$. In particular, $\braket{x}{x'}=\delta(x-x')$.

Before proceeding to estimate the error in the Green's function, we make some assumptions on the perturbation $\epsilon W$. First, we require that $\epsilon<\|H_{0}^{-1}\|$ in order to make the operator $H_{0}+\epsilon W$ invertible. 
In the following derivation, we need the inverse of $H_{0}$ to have finite off-diagonal rank. To understand this condition better, we write $H_{0}$ and $H_{0}^{-1}$ in the block form:
\[H_{0}=\left( \begin{array}{cc}
                  A & 0 \\
                  G & D \end{array}\right)\text{ and }
H_{0}^{-1}=\left( \begin{array}{cc}
                  E & C\\
                  F & H\end{array}\right),\]
where $A: L^{2}([0,x_{0}])\to L^{2}([0,x_{1}])$ is the restriction of $H_{0}$ and $0<x_{1}<x_{0}<1$. We are concerned with the Green's function $G(x,x')$ where $x\leq x_{0}$ and $x'\geq x_{1}$. We observe that $AC=0$. Hence, $\mathrm{Im}C\subset\mathrm{Ker}A$. Now we assume that $A$ is a Fredholm operator, or more generally, a bounded operator with finite-dimensional kernal and closed image. Then $C$ has finite rank. Moreover, since the index of a Fredholm operator is invariant under sufficiently small perturbations, we can absorb the diagonal part of $W$ into $H_{0}$ without increasing the rank of $C$. Without loss of generality, we can assume that $\bra{x'}W\ket{x''}\neq 0$ only if $x'\leq x_{1}<x_{0}\leq x''$, or $x''\leq x_{0}$ and $x'\geq x_{1}$. This is consistent with the assumption we have made in the discrete case.

The standard perturbation theory yields that
\begin{eqnarray*}
G(x,x')&=&\braket{x}{\psi}\\
&=&\bra{x}\frac{1}{H_{0}+\epsilon W}\ket{x'}\\
&=&\bra{x}(H_{0}^{-1}-\epsilon H_{0}^{-1}WH_{0}^{-1}+o(\epsilon^{2})\ket{x'}\\
&=&G_{0}(x,x')-\epsilon\int_{[0,1]\times[0,1]} G_{0}(x,x'')\bra{x''}W\ket{x'''}G_{0}(x''',x')dx''dx'''+o(\epsilon^{2}),
\end{eqnarray*}
where $G_{0}$ is the Green's function of the unperturbed operator $H_{0}$. We will denote $\bra{x''}W\ket{x'''}$ by $W(x'',x''')$ from now on.

By the assumption on $W$, $W(x'',x''')\neq0$ only if $x''\leq x_{1}<x_{0}\leq x'''$, or $x'''\leq x_{0}$ and $x''\geq x_{1}$. Hence, the integral can be broken into two parts:
\begin{eqnarray*}
&&\int_{[0,1]\times[0,1]}G_{0}(x,x'')W(x'',x''')G_{0}(x''',x')dx''dx'''\\
&=&\int_{0}^{x_{1}}dx''G_{0}(x,x'')\int_{x_{0}}^{1}dx'''W(x'',x''')G_{0}(x''',x')+\\
&&\int_{x_{1}}^{1}dx''G_{0}(x,x'')\int_{0}^{x_{0}}dx'''W(x'',x''')G_{0}(x''',x')
\end{eqnarray*}
The second term on the right hand side does not increase the rank since $x\leq x_{0}$ and $x''\geq x_{1}$. Furthermore, $\|W\|=1$, so the $L^{2}$ norm of the first term is bounded by $\frac{1}{\sigma_{1}\sigma_{2}}$, where 
\begin{eqnarray}
\sigma_{1}&=&(\int_{0}^{x_{0}}dy\int_{0}^{x_{1}}dy'|G_{0}(y,y')|^{2})^{-\frac{1}{2}};\\
\sigma_{2}&=&(\int_{x_{0}}^{1}dy\int_{x_{1}}^{1}dy'|G_{0}(y,y')|^{2})^{-\frac{1}{2}}.
\label{sigmas}
\end{eqnarray}
They are the continuous counterparts of $\sigma_{k}(A)$ and $\sigma_{n-k+p}(D)$ in the discrete case. Combining all the above equations, we conclude that if the off-diagonal rank of $H_{0}$ is $n$, then the $(n+1)th$ singular value of the off-diagonal block of $H_{0}+\epsilon W$ is bounded by $\frac{\epsilon}{\sigma_{1}\sigma_{2}}+o(\epsilon^{2})$.

\section{Open Question}
The Nullity Theorem also applies to the case where the unperturbed matrix has a low
off-diagonal rank, but not a banded matrix. We have assumed
throughout that $B$ in Theorem \eqref{nullity} equals zero. However,
if $\mathrm{rank}B=k\neq0$, we still have a low rank off-diagonal
block: $\mathrm{rank}C=p+k$, by the Nullity Theorem. However, it is
considerably harder to generalize this result to an approximate
case. In particular, there is no known estimate for
$\sigma_{p+k+1}(C)/\sigma_{k+1}(B)$. Experiments indicate that the
dependence of this ratio on $A$ and $D$ is very complicated and it
may also depend on the first $k$ singular values of $B$. Therefore,
our result is not a full generalization of the Nullity Theorem.
\section{Conclusion}
In this paper, we have studied the nearly-banded matrices via the
off-diagonal blocks of their inverses and obtained an estimate
(Corollary \eqref{estimate}) for the $(p+1)$th singular value of an
off-diagonal block, which in some sense measures the performance of
the banded approximation. The result is then extended to a particular example of function spaces. This research is sponsored by the Lord Fund in MIT. The research is initiated by Gilbert Strang, and the extension to the continuous case is inspired by Steven Johnson, to whom the author is deeply grateful.

\end{document}